\documentclass[12pt]{amsart}

\font\bbbld=msbm10 scaled\magstephalf
\newcommand{\bfH}{\hbox{\bbbld H}}
\newcommand{\bfR}{\hbox{\bbbld R}}

\newcommand{\vx}{{\bf x}}

\newcommand{\e}{\varepsilon}
\newcommand{\goto}{\rightarrow}
\newcommand{\ol}{\overline}

\newcommand{\be}{\begin{equation}}
\newcommand{\ee}{\end{equation}}

\newcommand{\tD}{{\tilde D}}

\newcommand{\tg}{{\tilde g}}
\newcommand{\thh}{{\tilde h}}

\newcommand{\tkappa}{{\tilde \kappa}}
\newcommand{\tnabla}{\tilde{\nabla}}

\newtheorem{theorem}{Theorem}[section]
\newtheorem{lemma}[theorem]{Lemma}
\newtheorem{proposition}[theorem]{Proposition}
\newtheorem{corollary}[theorem]{Corollary}

\theoremstyle{definition}

\newtheorem{example}[theorem]{Example}

\theoremstyle{remark}
\newtheorem{remark}[theorem]{Remark}

\numberwithin{equation}{section}

\begin{document}
\setlength{\baselineskip}{1.2\baselineskip}

\
\title[Hypersurfaces of Constant Curvature]
{
Hypersurfaces of Constant Curvature in Hyperbolic Space.}
\author{Bo Guan}
\address{Department of Mathematics, Ohio State University,
         Columbus, OH 43210}
\email{guan@math.osu.edu}
\author{Joel Spruck}
\address{Department of Mathematics, Johns Hopkins University,
 Baltimore, MD 21218}
\email{js@math.jhu.edu}
\thanks{Research of both authors was supported in part
by NSF grants.}

\begin{abstract}

\noindent
We show that for a very general and natural class of curvature functions,
the problem of finding a complete strictly convex hypersurface in $\bfH^{n+1}$ satisfying
$f(\kappa)=\sigma \in (0,1)$ with a prescribed asymptotic boundary $\Gamma$ at infinity
has at least one solution which is a ``vertical graph'' over the interior (or the exterior)
of  $\Gamma$. There is uniqueness for a certain subclass of these curvature functions which
includes the curvature quotients $(\frac{H_n}{H_l})^{\frac1{n-l}}$, $l-2,$ or $l-1$.
For smooth simple $\Gamma$, as $\sigma$ varies between 0 and 1,  these hypersurfaces foliate
the two components of  the complement of the hyperbolic convex hull of $\Gamma$.
\end{abstract}

\maketitle

\section{Introduction}
\label{sec1}
\setcounter{equation}{0}

In this paper we return to our earlier study \cite{GSS} of complete locally
strictly convex hypersurfaces
of constant  curvature in hyperbolic space $\bfH^{n+1}$ with a
prescribed asymptotic boundary at infinity. Given $\Gamma \subset
\partial_{\infty} \bfH^{n+1}$  and a smooth symmetric function $f$ of
$n$ variables, we seek a complete hypersurface $\Sigma$ in
$\bfH^{n+1}$ satisfying
\begin{equation}
\label{eq1.10}
f(\kappa[\Sigma]) = \sigma
\ee
\be \label{eq1.20}
\partial \Sigma = \Gamma
\end{equation}
where $\kappa[\Sigma] = (\kappa_1, \dots, \kappa_n)$
denotes the {\em positive} hyperbolic principal curvatures of $\Sigma$ and $\sigma \in (0,1)$
is a constant.

We will use the half-space model,
\[ \bfH^{n+1} = \{(x, x_{n+1}) \in \bfR^{n+1}: x_{n+1} > 0\} \]
equipped with the hyperbolic metric
\begin{equation}
\label{eq1.30}
 ds^2 = \frac{\sum_{i=1}^{n+1}dx_i^2}{x_{n+1}^2}.
\end{equation}
Thus $\partial_\infty \bfH^{n+1}$ is naturally identified with
$\bfR^n = \bfR^n \times \{0\} \subset \bfR^{n+1}$ and (\ref{eq1.20}) may
be understood in the Euclidean sense. For convenience we say $\Sigma$ has
compact asymptotic boundary if
$\partial \Sigma \subset \partial_\infty \bfH^{n+1}$ is compact with respect
to the Euclidean metric in $\bfR^n$.

The function $f$ is assumed to satisfy the fundamental structure
conditions in
\begin{equation}
\label{1.40}
K^+_n := \big\{\lambda \in \bfR^n:
   \mbox{each component $\lambda_i > 0$}\big\}:
 \end{equation}

\begin{equation}
\label{eq1.50}
f_i (\lambda) \equiv \frac{\partial f (\lambda)}{\partial \lambda_i} > 0
  \;\; \mbox{in $K^+_n$}, \;\; 1 \leq i \leq n,
\end{equation}
\begin{equation}
\label{eq1.60}
\mbox{$f$ is a concave function in $K^+_n$},
\end{equation}
and
\begin{equation}
\label{eq1.70}
 f > 0 \;\;\mbox{in $K^+_n$},
  \;\; f = 0 \;\;\mbox{on $\partial K^+_n$}
\end{equation}
 In addition, we shall assume that $f$ is normalized
\begin{equation}
\label{eq1.80}
f(1, \dots, 1) = 1
\end{equation}
and satisfies the following more technical assumptions
\begin{equation}
\label{eq1.90}
\mbox{ $f$ is homogeneous of degree one}
\end{equation}
and
\begin{equation}
\label{eq1.100}
\lim_{R \rightarrow + \infty}
   f (\lambda_1, \cdots, \lambda_{n-1}, \lambda_n + R)
    \geq 1 + \varepsilon_0 \;\;\;
\mbox{uniformly in $B_{\delta_0} ({\bf 1})$}
\end{equation}
for some fixed $\varepsilon_0 > 0$ and $\delta_0 > 0$,
where $B_{\delta_0} ({\bf 1})$ is the ball
of radius $\delta_0$ centered at ${\bf 1} = (1, \dots, 1) \in \bfR^n$.

All these assumptions are satisfied by $f =
(H_n/H_l)^{\frac{1}{n-l}}$, $0 \leq l <  n$, 
where $H_l$ is the normalized $l$-th elementary
symmetric polynomial ($H_0 =1$, $H_1=H$ and $H_n=K$ the
mean and extrinsic Gauss curvatures, respectively).
See \cite{CNS3} for proof of (\ref{eq1.50}) and (\ref{eq1.60}).
For  (\ref{eq1.100}) one easily computes that
\[\lim_{R \rightarrow + \infty}
   f (\lambda_1, \cdots, \lambda_{n-1}, \lambda_n + R)
  = \Big(\frac{n}l\Big)^{\frac1{n-l}}.\]
  Moreover, if
$g^{k},\, k=1,\ldots N$ satisfy (\ref{eq1.50})-(\ref{eq1.100}),
then so does the ``concave sum'' $f=\sum_{k=1}^N\alpha_k g^k$ or
``concave product'' $f=\Pi_{k=1}^N (g^k)^{\alpha_k}$ where
$\alpha_k>0,\, \sum_{k=1}^N \alpha_k=1$.\\

  Since $f$ is symmetric, by (\ref{eq1.60}),
(\ref{eq1.80}) and (\ref{eq1.90}) we have
\begin{equation}
\label{eq1.110}
f (\lambda) \leq f ({\bf 1}) + \sum f_i ({\bf 1}) (\lambda_i - 1)
= \sum f_i ({\bf 1}) \lambda_i  = \frac{1}{n} \sum \lambda_i
\;\;\mbox{in $K^+_n$}
\end{equation}
and
\begin{equation}
\label{eq1.120}
 \sum f_i (\lambda) = f (\lambda) + \sum f_i (\lambda) (1 - \lambda_i)
\geq f ({\bf 1}) = 1 \;\;\mbox{in $K^+_n$}.
\end{equation}

 In this paper all
hypersurfaces in $\bfH^{n+1}$ we consider are assumed to be
connected and orientable. If $\Sigma$ is a complete hypersurface in
$\bfH^{n+1}$ with compact asymptotic boundary at infinity, then the
normal vector field of $\Sigma$ is chosen to be the one pointing to
the unique unbounded region in $\bfR^{n+1}_+ \setminus \Sigma$, and
the (both hyperbolic and Euclidean) principal curvatures of $\Sigma$
are calculated with respect to this normal vector field.

As in our earlier work \cite{RS94, NS96, GS00, GSS, GS08}, we will take
 $\Gamma=\partial \Omega$ where
$\Omega \subset \bfR^n$ is a smooth domain and seek $\Sigma$ as  the
graph of  a function $u(x)$ over $\Omega$, i.e.
  \[\Sigma=\{(x,x_{n+1}): x\in \Omega,~ x_{n+1}=u(x)\}.\]
 Then  the coordinate vector fields and upper unit normal are  given by
 \[X_i=e_i+u_i e_{n+1},~{\bf n}=u\nu=u\frac{(-u_i e_i+e_{n+1})}{w},\]
 where $ w=\sqrt{1+|\nabla u|^2}$.
 The first fundamental form $g_{ij}$ is then given by
  \begin{equation}
\label{eq1.130}
 g_{ij} = \langle X_i,X_j \rangle
        = \frac1{u^2}(\delta_{ij} + u_i u_j)=\frac{g^e_{ij}}{u^2}~.
\end{equation}
To compute the second fundamental form $h_{ij}$ we use
\be \Gamma_{ij}^k=\frac1{x_{n+1}}\{-\delta_{jk}\delta_{i n+1}
          -\delta_{ik}\delta_{j n+1}+\delta_{ij}\delta_{k n+1}\}
\ee
to obtain
\be
\nabla_{X_i}X_j
 =(\frac{\delta_{ij}}{x_{n+1}}+u_{ij}-\frac{u_i u_j}{x_{n+1}})e_{n+1}
   -\frac{u_j e_i+u_i e_j}{x_{n+1}}~.
\ee Then
\begin{equation}
\label{eq1.140}
\begin{aligned}
 h_{ij} & =  \, \langle \nabla_{X_i}X_j,u\nu \rangle =
  \frac1{uw}(\frac{\delta_{ij}}u+ u_{ij}-\frac{u_i u_j}u+2\frac{u_i u_j}u)\\
        & =  \, \frac1{u^2 w}{(\delta_{ij}+u_i u_j +u u_{ij})}
          =  \, \frac{h^e_{ij}}{u}+\frac{\nu^{n+1}}{u^2}g^e_{ij}.
 \end{aligned}
 \end{equation}
The hyperbolic principal curvatures $\kappa_i$ of $\Sigma$ are the
roots of the characteristic equation
\[\det(h_{ij}-\kappa g_{ij})
    =u^{-n}\det(h^e_{ij}-\frac1u(\kappa-\frac1w)g^e_{ij})=0.\]
Therefore,
 \be
 \label{eq1.150}
 \kappa_i=u\kappa^e_i +\nu^{n+1}.
 \ee

The relations \eqref{eq1.140} and \eqref{eq1.150}  are easily seen
to hold for parametric hypersurfaces.\\

One beautiful consequence of \eqref{eq1.140} is the following result of \cite{GSS}.
\begin{theorem}
\label{th1.0}
Let $\Sigma$ be a complete locally strictly convex $C^2$ hypersurface in
$\bfH^{n+1}$ with compact asymptotic boundary at infinity.
Then $\Sigma$ is the (vertical) graph of a function
$u \in C^2 (\Omega) \cap C^0 (\ol{\Omega})$, $u > 0$ in $\Omega$
and  $u = 0$ on $\ol{\Omega}$, for some domain $\Omega \subset \bfR^n$:
\[ \Sigma = \big\{(x, u (x)) \in \bfR^{n+1}_+: x \in \Omega\big\} \]
such that
\begin{equation}
\label{eq1.160}
 \{\delta_{ij} + u_i u_j + u u_{ij}\} > 0 \;\; \mbox{in $\Omega$.}
\end{equation}
That is, the function $u^2 + |x|^2$ is strictly convex.
\end{theorem}

According to Theorem~\ref{th1.0}, our assumption that $\Sigma$ is a graph
is completely general and the asymptotic boundary $\Gamma$ must be the boundary
of some bounded domain $\Omega$ in $\bfR^n$. \\

 Problem (\ref{eq1.10})-(\ref{eq1.20}) then reduces to the
Dirichlet problem for a fully nonlinear second order equation which we
shall write in the form
\begin{equation}
\label{eq1.170}
G(D^2u, Du, u) = \sigma,
\;\; u > 0 \;\;\; \text{in $\Omega \subset \bfR^n$}
\end{equation}
with the boundary condition
\begin{equation}
\label{eq1.180}
             u = 0 \;\;\;    \text{on $\partial \Omega$}.
\end{equation}


We seek solutions of equation (\ref{eq1.170}) satisfying
(\ref{eq1.160}). Following the literature we call such solutions
{\em admissible}. By \cite{CNS3} condition~(\ref{eq1.50}) implies
that equation (\ref{eq1.170}) is elliptic for admissible solutions.
Our goal is to show that the Dirichlet problem
(\ref{eq1.170})-(\ref{eq1.180}) admits smooth admissible solutions
for all $0 < \sigma < 1$, which is optimal.\\

Our main result of the paper may be stated as follows.

\begin{theorem}
\label{th2.0}
Let $\Gamma = \partial \Omega \times \{0\} \subset \bfR^{n+1}$
where $\Omega$ is a bounded smooth domain in $\bfR^n$.
Suppose $\sigma \in (0, 1)$  and that $f$ satisfies
 conditions (\ref{eq1.50})-(\ref{eq1.100}) in $K_n^+$.
Then there exists a complete locally strictly convex hypersurface
$\Sigma$ in $\bfH^{n+1}$ satisfying (\ref{eq1.10})-(\ref{eq1.20})
with uniformly bounded principal curvatures
\begin{equation}
\label{eq1.190}
 |\kappa[\Sigma]| \leq C \;\; \mbox{on $\Sigma$}.
\end{equation}
Moreover, $\Sigma$ is the graph of an admissible solution $u \in C^\infty
(\Omega) \cap C^1 (\bar{\Omega})$ of the Dirichlet problem
\eqref{eq1.170}-\eqref{eq1.180}.
Furthermore, $u^2 \in C^{\infty} (\Omega)\cap C^{1,1}(\ol{\Omega}) $
and
\begin{equation}
\label{eq1.200}
\begin{aligned}
&\, u|D^2 u| \leq C
\;\;\; \mbox{in $\Omega$}, \\
&\, \sqrt{1 + |Du|^2} = \frac{1}{\sigma}
\;\;\; \mbox{on $\partial \Omega$}
\end{aligned}
\end{equation}
\end{theorem}

For Gauss curvature, $f (\lambda)=(H_n)^{\frac1n}$,  Theorem \ref{th2.0}
was proved by Rosenberg and Spruck \cite{RS94}.

Equation (\ref{eq1.170}) is
singular where $u = 0$. It is therefore natural to approximate the
boundary condition (\ref{eq1.180}) by
\begin{equation}
\label{eq1.180'}
             u = \epsilon > 0 \;\;\;  \text{on $\partial \Omega$}.
\end{equation}
When $\epsilon$ is sufficiently small, we showed in \cite{GSS} that the
Dirichlet problem (\ref{eq1.170}),(\ref{eq1.180'}) is solvable for all
$\sigma \in (0, 1)$.

\begin{theorem}
\label{th3.0}
Let $\Omega$ be a bounded smooth domain in $\bfR^n$
and $\sigma \in (0, 1)$. Suppose $f$ satisfies
(\ref{eq1.50})-(\ref{eq1.100}) in $K_n^+$. Then for any
$\epsilon > 0$ sufficiently small, there exists an admissible
solution $u^{\epsilon} \in C^\infty (\bar{\Omega})$ of the Dirichlet
problem (\ref{eq1.170}), (\ref{eq1.180'}). Moreover, $u^{\epsilon}$
satisfies the {\em a priori} estimates
\begin{equation}
\label{eq1.210}
\sqrt{1 + |D u^{\epsilon}|^2} \leq \frac{1}{\sigma} + C \epsilon, \,
   u^{\epsilon}|D^2 u^{\epsilon}| \leq C
\;\;\; \mbox{on $\partial \Omega$}\, ,
\end{equation}
and
\begin{equation}
\label{eq1.220}
u^{\epsilon}|D^2 u^{\epsilon}| \leq \frac{C}{\epsilon^2}
\;\;\; \mbox{in $\Omega$.}
\end{equation}
where $C$ is independent of $\epsilon$.
\end{theorem}

\begin{remark}
The global gradient estimate, Lemma 3.4 of \cite{GSS} is not correct as stated.
This may be corrected using the convexity argument of \cite{RS94} or using
Corollary~\ref{cor3.3} of section~\ref{sec3} of this paper.
Theorem~\ref{th3.0} above as well as Theorem 1.2 of \cite{GSS} remain valid.
However no apriori uniqueness can be asserted. In Theorem~\ref{th5.0} we
prove a uniqueness result for a special class of curvature functions.
\end{remark}

Our main technical difficulty in proving Theorem \ref{th2.0} is that the
estimate \eqref{eq1.220} does not allow us to pass to the limit.
In \cite{GSS} we were able to obtain a global estimate independent of
$\epsilon$ for the hyperbolic principal curvatures for $\sigma^2>\frac18$.
In this paper we obtain such estimates for all $\sigma \in (0,1)$ by proving
a maximum principle for the largest hyperbolic principal curvature.

\begin{theorem}
\label{th4.0}
Let $\Omega$ be a bounded smooth domain in $\bfR^n$ and $\sigma \in (0, 1)$.
Suppose $f$ satisfies (\ref{eq1.50})-(\ref{eq1.100}) in $K_n^+$.
Then for any admissible solution $u^{\epsilon} \in C^\infty (\bar{\Omega})$
of the Dirichlet problem (\ref{eq1.170}), (\ref{eq1.180'}),
\be
\label{eq1.230}
\max_{\vx \in \Sigma^{\epsilon}} \kappa_{\max} (\vx)
  \leq C(1+\max_{\vx \in \partial  \Sigma^{\epsilon}} \kappa_{\max} (\vx) )
\ee
where $\Sigma^{\epsilon}=\mbox{graph $u^{\epsilon}$}$ and C is independent
of $\epsilon$.
\end{theorem}

By Theorem \ref{th4.0}, the hyperbolic  principal curvatures of the
admissible solution $u^{\epsilon}$ given in Theorem \ref{th3.0} are
uniformly bounded above independent of $\epsilon$.
Since $f(\kappa[u^{\epsilon}])=\sigma$ and $f=0$ on $\partial K_n^+$,
the hyperbolic principal curvatures admit a uniform positive lower bound
independent of $\epsilon$ and therefore (\ref{eq1.170}) is uniformly
elliptic on compact subsets of $\Omega$ for the solution $u^{\epsilon}$.
By the interior estimates of Evans and Krylov, we obtain uniform
$C^{2,\alpha}$ estimates for any compact subdomain of $\Omega$.
The proof of Theorem \ref{th2.0} is now routine. \\

Finally we prove a uniqueness result and as an application prove a result
about foliations. This latter result is relevant to the study of foliations
of the complement of the convex core of quasi-fuchsian manifolds
(see \cite{Labourie}, \cite{RS94}, \cite{Smith}).

\begin{theorem}
\label{th5.0}
Suppose $f$ satisfies
(\ref{eq1.50})-(\ref{eq1.100}) in $K_n^+$ and in addition,
\be \label{eq1.240}
\sum f_i > \sum \lambda_i^2 f_i \,\,\mbox {in $K_n^+ \cap \{ 0<f<1\}$}.
\ee
Then the solutions given in Theorem \ref{th2.0} and Theorem \ref{th3.0} are
unique. In particular uniqueness holds for
$f=(\frac{H_n}{H_l})^{\frac1{n-l}}$ with $l-1$ or $l-2$.
\end{theorem}

\begin{theorem}
\label{th6.0}
a. Let f satisfy the conditions of Theorem \ref{th5.0} and assume that
$\Gamma$ is smooth. Then for each $\sigma \in (0,1)$ there are exactly
two embedded strictly locally convex hypersurfaces satisfying
\eqref{eq1.10}, \eqref{eq1.20}.
Each surface is a graph of
$u^{\sigma} \in C^{\infty}(\Omega^{\pm}) \cap C^{1}(\ol{\Omega^{\pm}})$
where $\Omega^{\pm}$ are the components of the complement of $\Gamma$.
Moreover the solution hypersurfaces
$\Sigma^{\sigma}= \mbox{graph\,\,$u^{\sigma}$}$ have uniformly bounded
principal curvatures and foliate each component of
$\bfH^{n+1} \setminus \mathcal{CH}(\Gamma)$, the complement of the hyperbolic
convex hull of $\Gamma$.

b. Let $f=\frac{H_n}{H_{n-1}}$ and $\Gamma=\partial \Omega$ where
$\Omega$ is a simply connected Jordan domain.
If $n>2$, assume  in addition that $\Gamma$ is regular for Laplace's equation.
Then the conclusions of part a. hold with
$u^{\sigma}(x) \in C^{\infty}(\Omega^{\pm}) \cap C^{0}(\ol{\Omega^{\pm}})$
and the principal curvatures are uniformly bounded on compact subsets.
\end{theorem}

\begin{remark}
i. Graham Smith pointed out to us that in the special case
$n=3$, $f=(\frac{K}H)^{\frac12}$ is his special Lagrangian curvature
with angle $\theta=\pi$ and interior curvature bounds follow from the
geometric ideas of his paper \cite{Smith2}.
Moreover in Lemma 7.4 of \cite{Smith} he showed that special Lagrangian
curvature with angle $\theta \geq (n-1)\frac{\pi}2$ satisfies our uniqueness
condition \eqref{eq1.240}. Thus by Theorems \ref{th5.0} and \ref{th6.0},
the existence of foliations of constant special Lagrangian curvature can be
proven for $\theta \geq (n-1)\frac{\pi}2$ for all $n$.
This includes the special
case $f=K^{\frac12}$ when $n=2$ and $ f=(\frac{K}H)^{\frac13}$ for $n=3$
mentioned above. See \cite{Smith3} for Graham Smith's most recent work which has
some overlap with ours.

ii. Rosenberg and Spruck~\cite{RS94} proved part b of Theorem~\ref{th6.0}
for $f=K^{\frac12}$ in case $n=2$. Here and also in \cite{RS94}, no claim
is made about the higher regularity of the $\mathcal{CH}(\Gamma)$.
In other words, the curvature estimates obtained in the proof of
Theorem~\ref{th6.0} (global for $\Gamma$ smooth and interior for $\Gamma$
Jordan) blow up as $\sigma \goto 0$. We have not yet derived interior
curvature estimates for the case $f=\frac{H_n}{H_{n-2}}$ in the general case.
\end{remark}

The organization of the paper is as follows.
In section \ref{sec2} we establish some basic identities
on a hypersurface $\Sigma$ satisfying \eqref{eq1.10} that will form the basis
of the global gradient estimates derived in section \ref{sec3} and the maximum
principle for $\kappa_{\max}$, the largest principal curvature of $\Sigma$,
which is carried out in section \ref{sec4}.
Finally in section \ref{sec5} we prove the uniqueness Theorem~\ref{th5.0} and
the foliation Theorem~\ref{th6.0}.

\section{Formulas on hypersurfaces}
\label{sec2}
\setcounter{equation}{0}

In this section we will derive some basic identities on a hypersurface
by comparing the induced hyperbolic and Euclidean metrics.\\

Let $\Sigma$ be a hypersurface in $\bfH^{n+1}$. We shall use $g$ and $\nabla$
to denote the induced hyperbolic metric and Levi-Civita connections
on $\Sigma$, respectively. As $\Sigma$ is also a submanifold of $\bfR^{n+1}$,
we shall usually distinguish a geometric quantity with respect to the
Euclidean metric by adding a `tilde' over the corresponding hyperbolic
quantity.
For instance, $\tg$ denotes the induced  metric on $\Sigma$
from $\bfR^{n+1}$, and $\tnabla$ is its Levi-Civita connection.

Let $\vx$ be the position vector of $\Sigma$ in $\bfR^{n+1}$ and set
\[ u = \vx \cdot \bf{e} \]
where $\bf{e}$ is the unit vector in the positive
$x_{n+1}$ direction in $\bfR^{n+1}$, and `$\cdot$' denotes the Euclidean
inner product in $\bfR^{n+1}$.
We refer $u$ as the {\em height function} of $\Sigma$.

Throughout the paper we assume $\Sigma$ is orientable and let ${\bf n}$ be
a (global)
unit normal vector field to $\Sigma$ with respect to the hyperbolic metric.
This also determines a unit normal $\nu$ to $\Sigma$ with respect to the
Euclidean metric by the relation
\[ \nu = \frac{{\bf n}}{u}. \]
We denote $ \nu^{n+1} = \bf{e} \cdot \nu$.

Let $(z_1, \ldots, z_n)$ be local coordinates and
\[ \tau_i = \frac{\partial}{\partial z_i}, \;\; i = 1, \ldots, n. \]
The hyperbolic and Euclidean metrics of $\Sigma$ are given by
\[ g_{ij} = \langle \tau_i, \tau_j \rangle, \;\;
   \tg_{ij} = \tau_i \cdot \tau_j = u^2 g_{ij}, \]
while the second fundamental forms are
\begin{equation}
\label{eq2.10}
 \begin{aligned}
   h_{ij} \,& = \langle D_{\tau_i} \tau_j, {\bf n} \rangle
              = - \langle D_{\tau_i} {\bf n}, \tau_j \rangle, \\
\thh_{ij} \,& = \nu \cdot \tD_{\tau_i} \tau_j
              = - \tau_j \cdot \tD_{\tau_i} \nu,
    \end{aligned}
\end{equation}
where $D$ and $\tD$ denote the Levi-Civita connection of $\bfH^{n+1}$
and $\bfR^{n+1}$, respectively.
 The following relations are well known (see \eqref{eq1.140}, \eqref{eq1.150}):
\begin{equation}
\label{eq2.20}
 h_{ij} = \frac{1}{u} \thh_{ij} + \frac{\nu^{n+1}}{u^2} \tg_{ij}.
 \end{equation}
and
 $\tilde{\kappa}_1, \cdots, \tilde{\kappa}_n$ by the formula
\begin{equation}
\label{eq2.30}
\kappa_i = u \tilde{\kappa}_i + \nu^{n+1}, \;\;\; i = 1, \cdots, n.
\end{equation}

The Christoffel symbols are related by the formula
\begin{equation}
\label{eq2.40}
\Gamma_{ij}^k = \tilde{\Gamma}_{ij}^k - \frac{1}{u}
   (u_i \delta_{kj} + u_j \delta_{ik} - \tg^{kl} u_l \tg_{ij}).
\end{equation}
It follows that for $v \in C^2 (\Sigma)$
\begin{equation}
\label{eq2.50}
\nabla_{ij} v = v_{ij} - \Gamma_{ij}^k v_k
  = \tilde{\nabla}_{ij} v + \frac{1}{u}
    (u_i v_j + u_j v_i - \tg^{kl} u_k v_l \tg_{ij})
\end{equation}
where (and in sequel)
\[ v_i = \frac{\partial v}{\partial x_i}, \;
   v_{ij} = \frac{\partial^2 v}{\partial x_i x_j}, \; \mbox{etc.} \]
In particular,
\begin{equation}
\label{eq2.60}
\begin{aligned}
\nabla_{ij} u
 \,& = \tilde{\nabla}_{ij} u + \frac{2 u_i u_j}{u}
       - \frac{1}{u} \tg^{kl} u_k u_l \tg_{ij}
  \end{aligned}
\end{equation}
and
\begin{equation}
\label{eq2.70}
\nabla_{ij} \frac{1}{u}
  = - \frac{1}{u^2} \tilde{\nabla}_{ij} u
  + \frac{1}{u^3} \tg^{kl} u_k u_l \tg_{ij}.
\end{equation}
Moreover,
\begin{equation}
\label{eq2.80}
\begin{aligned}
\nabla_{ij} \frac{v}{u}
 \,&  = v \nabla_{ij} \frac{1}{u}
        + \frac{1}{u} \tilde{\nabla}_{ij} v
        - \frac{1}{u^2} \tg^{kl} u_k v_l \tg_{ij}.
 \end{aligned}
\end{equation}

In $\bfR^{n+1}$,
\begin{equation}
\label{eq2.90}
\begin{aligned}
             \tg^{kl} u_k u_l
      \,& = |\tilde{\nabla} u|^2 = 1 - (\nu^{n+1})^2 \\
            \tilde{\nabla}_{ij} u
      \,& = \thh_{ij} \nu^{n+1}.
 \end{aligned}
\end{equation}
Therefore, by \eqref{eq2.30} and \eqref{eq2.70},
\begin{equation}
\label{eq2.100}
\begin{aligned}
\nabla_{ij} \frac{1}{u}
 \,&  = - \frac{\nu^{n+1}}{u^2} \thh_{ij}
  + \frac{1}{u^3} (1 - (\nu^{n+1})^2) \tg_{ij} \\
 \,&  = \frac{1}{u} (g_{ij} - \nu^{n+1} h_{ij}).
 \end{aligned}
\end{equation}
We note that \eqref{eq2.80} and \eqref{eq2.100} still hold for
general local frames $\tau_1, \ldots, \tau_n$.
In particular, if $\tau_1, \ldots, \tau_n$ are orthonormal in the hyperbolic
metric, then
$g_{ij} = \delta_{ij}$ and $\tg_{ij} = u^2 \delta_{ij}$.

We now consider equation~\eqref{eq1.10} on $\Sigma$.
Let $\mathcal{A}$ be the vector space of $n \times n$ matrices and
\[ \mathcal{A}^+= \{A = \{a_{ij}\} \in \mathcal{A}: \lambda (A) \in K_n^+\}, \]
where $\lambda (A) = (\lambda_1, \dots, \lambda_n)$ denotes the eigenvalues of $A$.
Let $F$ be the function defined by
\begin{equation}
\label{eq2.110}
F (A) = f (\lambda (A)), \;\; A \in \mathcal{A}^+
\end{equation}
and denote
\begin{equation}
\label{eq2.120}
F^{ij} (A) = \frac{\partial F}{\partial a_{ij}} (A), \;\;
  F^{ij, kl} (A) = \frac{\partial^2 F}{\partial a_{ij} \partial a_{kl}} (A).
\end{equation}
Since $F (A)$ depends only on the eigenvalues of $A$, if $A$ is symmetric
then so is the matrix $\{F^{ij} (A)\}$. Moreover,
\[ F^{ij} (A) = f_i \delta_{ij} \]
when $A$ is diagonal, and
\begin{equation}
\label{eq2.130}
 F^{ij} (A) a_{ij} = \sum f_i (\lambda (A)) \lambda_i = F (A),
\end{equation}
\begin{equation}
\label{eq2.140}
F^{ij} (A) a_{ik} a_{jk} = \sum f_i (\lambda (A)) \lambda_i^2.
\end{equation}

Equation~\eqref{eq1.10} can therefore be rewritten in a  local
frame $\tau_1, \ldots, \tau_n$ in the form
\begin{equation}
\label{eq2.150}
F (A[\Sigma]) = \sigma
\end{equation}
where $A[\Sigma] = \{g^{ik} h_{kj}\}$.
Let $F^{ij} = F^{ij} (A[\Sigma])$, $F^{ij, kl} = F^{ij, kl} (A[\Sigma])$.

\begin{lemma}
\label{lem2.10}
Let $\Sigma$ be a smooth hypersurface in $\bfH^{n+1}$ satisfying
equation~\eqref{eq1.10}. Then in a local orthonormal frame,
\begin{equation}
\label{eq2.160}
  F^{ij} \nabla_{ij} \frac{1}{u}
    = - \frac{\sigma \nu^{n+1}}{u} + \frac{1}{u} \sum f_i.
\end{equation}
and
\begin{equation}
\label{eq2.170}
  F^{ij} \nabla_{ij} \frac{\nu^{n+1}}{u} =
\frac{\sigma}{u} - \frac{\nu^{n+1}}{u} \sum f_i \kappa_i^2.
\end{equation}
\end{lemma}

\begin{proof}
The first identity follows immediately from \eqref{eq2.100},
\eqref{eq2.130} and assumption \eqref{eq1.90}. To prove
\eqref{eq2.170} we recall the identities in $\bfR^{n+1}$
\begin{equation}
\label{eq2.180}
 \begin{aligned}
        (\nu^{n+1})_i
  \,& = - \thh_{ij} \tg^{jk} u_k, \\
        \tilde{\nabla}_{ij} \nu^{n+1}
  \,& = - \tg^{kl} (\nu^{n+1} \thh_{il} \thh_{kj} + u_l \tnabla_k \thh_{ij}).
\end{aligned}
\end{equation}
By \eqref{eq2.20}, \eqref{eq2.130}, \eqref{eq2.140}, and
$\tg^{ik} = \delta_{jk}/u^2$ we see that
\begin{equation}
\label{eq2.190}
 \begin{aligned}
 F^{ij} \tg^{kl} \thh_{il} \thh_{kj}
   = \,& \frac{1}{u^2} F^{ij} \thh_{ik} \thh_{kj} \\
   = \,& F^{ij} (h_{ik} h_{kj} - 2 \nu^{n+1} h_{ij} + (\nu^{n+1})^2 \delta_{ij}) \\
   = \,& f_i \kappa_i^2 - 2 \nu^{n+1} \sigma + (\nu^{n+1})^2 \sum f_i.
\end{aligned}
\end{equation}

As a hypersurface in $\bfR^{n+1}$,
it follows from \eqref{eq2.30} that $\Sigma$ satisfies
\[ f (u \tkappa_1 + \nu^{n+1}, \ldots, u \tkappa_n + \nu^{n+1}) = \sigma, \]
or equivalently,
\begin{equation}
\label{eq2.200}
F (\{\tg^{ik} (u \thh_{kj} + \nu^{n+1} \tg_{kj})\}) = \sigma.
\end{equation}
Differentiating equation~\eqref{eq2.200}
and using $\tg_{ik} = u^2 \delta_{ik}$, $\tg^{ik} = \delta_{ik}/u^2$, we obtain
\begin{equation}
\label{eq2.210}
 F^{ij} (u \tnabla_k \thh_{ij} + u_k \thh_{ij}
     + (\nu^{n+1})_k u^2 \delta_{ij}) = 0.
\end{equation}
That is,
\begin{equation}
\label{eq2.220}
\begin{aligned}
F^{ij} \tnabla_k \thh_{ij} + \,& (\nu^{n+1})_k u \sum F^{ii}
   = - \frac{u_k}{u} F^{ij} \thh_{ij}  \\
   = \,& - u_k F^{ij} (h_{ij} - \nu^{n+1} \delta_{ij}) \\
   = \,& - u_k \Big(\sigma - \nu^{n+1} \sum f_i\Big).
  \end{aligned}
\end{equation}

Finally, combining \eqref{eq2.80}, \eqref{eq2.160}, \eqref{eq2.180},
\eqref{eq2.190}, \eqref{eq2.220}, and the first identity in \eqref{eq2.90},
 we derive
\begin{equation}
\label{eq2.230}
\begin{aligned}
  F^{ij} \nabla_{ij} \frac{\nu^{n+1}}{u}
  = \,& \nu^{n+1} F^{ij} \nabla_{ij} \frac{1}{u}
        + \frac{|\tnabla u|^2}{u} F^{ij} \thh_{ij}
        - \frac{\nu^{n+1}}{u^3} F^{ij} \thh_{ik} \thh_{kj} \\
  = \,& \frac{\nu^{n+1}}{u} \Big(\sum f_i- \nu^{n+1} \sigma\Big)
        + \frac{|\tnabla u|^2}{u} \Big(\sigma - \nu^{n+1} \sum f_i\Big)  \\
    \,& - \frac{\nu^{n+1}}{u} \Big(f_i \kappa_i^2
     - 2 \nu^{n+1} \sigma + (\nu^{n+1})^2 \sum f_i\Big) \\
  = \,& \frac{\sigma}{u} - \frac{\nu^{n+1}}{u} \sum f_i \kappa_i^2.
  \end{aligned}
\end{equation}
This proves \eqref{eq2.170}.

\end{proof}

\section{The asymptotic angle maximum principle and gradient estimates}
\label{sec3}
\setcounter{equation}{0}

In this section we show that the upward unit normal of a solution tends to
a fixed asymptotic angle on approach to the boundary. This implies a global
gradient bound on solutions.\\

\begin{theorem}
\label{th3.1}
Let $\Sigma$ be a smooth strictly locally convex hypersurface in $\bfH^{n+1}$
satisfying equation~\eqref{eq1.10}. Suppose $\Sigma$ is globally a graph:
\[ \Sigma = \{(x, u (x)): x \in \Omega\} \]
where $\Omega$ is a domain in $\bfR^n \equiv \partial \bfH^{n+1}$. Then
\be
\label{eq3.05}
F^{ij}\nabla_{ij}\frac{\sigma - \nu^{n+1}}{u}
   \geq \sigma(1-\sigma)\frac{(\sum f_i -1)}u  \geq 0
\ee
and so,
\begin{equation}
\label{eq3.10}
\frac{\sigma - \nu^{n+1}}{u}
  \leq \sup_{\partial \Sigma} \frac{\sigma - \nu^{n+1}}{u}
\;\; \mbox{on $\Sigma$}.
\end{equation}
Moreover, if $u = \epsilon>0$ on  $\partial \Omega$, then there exists
$\epsilon_0 > 0$ depending only on $\partial \Omega$, such that for all
$\epsilon \leq \epsilon_0$,
\begin{equation}
\label{eq3.20}
\frac{\sigma - \nu^{n+1}}{u}
   \leq \frac{\sqrt{1-\sigma^2}}{r_1}+\frac{\e(1+\sigma)}{r_1^2}
\;\; \mbox{on $\Sigma$}
\end{equation}
where $r_1$ is the maximal radius of exterior tangent spheres to
$\partial \Omega$.
\end{theorem}

\begin{proof}
Set $\eta= \frac{\sigma - \nu^{n+1}}{u}$.
By \eqref{eq2.160} and \eqref{eq2.170} we have
\[  F^{ij} \nabla_{ij} \eta
    = \frac{\sigma}{u} \Big(\sum f_i - 1\Big)
    + \frac{\nu^{n+1}}{u} \Big(\sum f_i \kappa_i^2 - \sigma^2\Big). \]

 On the other hand,
\[ \sum \kappa_i^2 f_i
    \geq \frac{(\sum \kappa_i f_i)^2}{\sum f_i}=\frac{\sigma^2}{\sum f_i}.\]
Hence,
\[ F^{ij}\nabla_{ij}\eta \geq \frac{\sigma}{u} \Big(\sum f_i - 1\Big)
               \Big(1-\frac{\sigma \nu^{n+1}}{\sum f_i}\Big)
\geq \frac{\sigma(1-\sigma)}{u} \Big(\sum f_i - 1\Big) \geq 0.\]

So \eqref{eq3.10} follows from the maximum principle, while \eqref{eq3.20}
follows from \eqref{eq3.10} and the approximate asymptotic angle condition,
\[\eta \leq \frac{\sqrt{1-\sigma^2}}{r_1}+\frac{\e(1+\sigma)}{r_1^2}
\;\; \mbox{on $\partial \Sigma$}\]
which is proved in Lemma~3.2 of \cite{GSS}.
\end{proof}

\begin{proposition}
\label{prop3.2}
Let $\Sigma$ be a smooth strictly locally convex graph
\[ \Sigma = \{(x, u (x)): x \in \Omega\} \]
in $\bfH^{n+1}$ satisfying $u\geq \e \, \mbox{in $\Omega$},\,
  u=\e \,\mbox{on $\partial \Omega$}$.
 Then
\be
\label{eq3.30}
 \frac1{ \nu^{n+1}} \leq \max \Big\{\frac{\max_{\Omega}u}u,
                       \max_{\partial \Omega}\frac 1{\nu^{n+1}}\Big\}.
\ee
\end{proposition}

\begin{proof}
Let $h=\frac{u}{\nu^{n+1}}=uw$ and suppose that $h$ assumes its maximum at
an interior point $x_0$. Then at $x_0$,
\[ \partial_i h = u_i w+u \frac{u_k u_{ki}}w
                = (\delta_{ki}+u_k u_i+u u_{ki})\frac{u_k}w = 0
\;\; \forall \; 1 \leq i \leq n.\]
Since $\Sigma$ is strictly locally convex, this implies that $\nabla u=0$
at $x_0$ so the proposition follows immediately.
\end{proof}

Combining Theorem~\ref{th3.1} and Proposition~\ref{prop3.2} gives

\begin{corollary}
\label{cor3.3}
Let $\Omega$ be a bounded smooth domain in $\bfR^n$
and $\sigma \in (0, 1)$. Suppose $f$ satisfies
(\ref{eq1.50})-(\ref{eq1.100}) in $K_n^+$. Then for any
$\epsilon > 0$ sufficiently small, any admissible
solution $u^{\epsilon} \in C^\infty (\bar{\Omega})$ of the Dirichlet
problem (\ref{eq1.170}),(\ref{eq1.180'}) satisfies the {\em apriori} estimate
\be
\label{eq3.40}
|\nabla u^{\epsilon}| \leq C
\;\;\; \mbox{in $\Omega$}
\ee
where $C$ is independent of $\epsilon$.
\end{corollary}

\section{Curvature estimates}
\label{sec4}
\setcounter{equation}{0}

In this section we prove a maximum principal for the largest principal
curvature of locally strictly convex graphs satisfying $f(\kappa)=\sigma$. \\

Let $\Sigma$ be a smooth hypersurface in $\bfH^{n+1}$ satisfying
$f(\kappa)=\sigma$.
For a fixed point $\vx_0 \in \Sigma$ we choose a local orthonormal frame
$\tau_1, \ldots, \tau_n$
around $\vx_0$ such that $h_{ij} (\vx_0) = \kappa_i \delta_{ij}$.
The calculations below are done at $\vx_0$. For convenience we shall
write $v_{ij} = \nabla_{ij} v$,  $h_{ijk} = \nabla_k h_{ij}$,
$h_{ijkl} = \nabla_{lk} h_{ij} = \nabla_l \nabla_k h_{ij}$, etc.\\

Since $\bfH^{n+1}$ has constant sectional curvature $-1$, by the Codazzi
and Gauss equations we have $ h_{ijk} = h_{ikj}$ and
\begin{equation}
\label{eq4.10}
\begin{aligned}
 h_{iijj} = \,& h_{jjii} + (h_{ii} h_{jj} - 1) (h_{ii} - h_{jj}) \\
          = \,& h_{jjii} + (\kappa_i \kappa_j - 1) (\kappa_i - \kappa_j).
 \end{aligned}
 \end{equation}
Consequently for each fixed $j$,
\begin{equation}
\label{eq4.20}
 F^{ii} h_{jjii} =  F^{ii} h_{iijj} + (1 + \kappa_j^2) \sum f_i \kappa_i
     - \kappa_j \sum f_i -  \kappa_j \sum \kappa_i^2  f_i.
\end{equation}

\begin{theorem}
\label{th4.10}
Let $\Sigma$ be a smooth strictly locally convex graph in $\bfH^{n+1}$
satisfying $f(\kappa)=\sigma$ and
\begin{equation}
\label{eq4.30}
\nu^{n+1} \geq 2 a > 0 \; \mbox{on $\Sigma$}.
\end{equation}
For $\vx \in \Sigma$ let $\kappa_{\max} (\vx)$ be the largest principal
curvature of $\Sigma$ at $\vx$. Then
\begin{equation}
\max_{\Sigma} \, \frac{\kappa_{\max}}{\nu^{n+1}-a} \leq
  \max \Big\{\frac{4n}{a^3}, \max_{\partial \Sigma} \,
                   \frac{\kappa_{\max}}{\nu^{n+1}-a}\Big\}.
\end{equation}
\end{theorem}

\begin{proof}
Let
\begin{equation}
\label{eq4.35}
 M_0 = \max_{\vx \in \Sigma} \frac{\kappa_{\max} (x) }{\nu^{n+1} - a}.
\end{equation}
Assume $M_0 > 0$ is attained at an interior point $\vx_0 \in \Sigma$.
Let $\tau_1, \ldots, \tau_n$ be a local orthonormal frame
around $\vx_0$ such that $h_{ij} (\vx_0) = \kappa_i \delta_{ij}$,
where $\kappa_1, \ldots, \kappa_n$ are the principal curvatures of $\Sigma$
at $\vx_0$. We may assume
$\kappa_1 = \kappa_{\max} (\vx_0)$.
Thus, at $\vx_0,\, \frac{h_{11}}{\nu^{n+1}-a}$ has a local maximum. Therefore,
\begin{eqnarray}
\label{eq4.40}
\frac{h_{11i}}{h_{11}} - \frac{\nabla_i \nu^{n+1}}{\nu^{n+1}-a} = 0, \\
\label{eq4.45}
\frac{h_{11ii}}{h_{11}}-\frac{\nabla_{ii} \nu^{n+1}}{\nu^{n+1}-a}\leq 0.
\end{eqnarray}

Using \eqref{eq4.20}, we find after differentiating the equation
$F(h_{ij})=\sigma$ twice that
\begin{lemma}
\label{lem4.20}
At $\vx_0$,
\begin{equation}
\label{eq4.50}
F^{ii}h_{11ii}= - F^{ij,rs}h_{ij1} h_{rs1} + \sigma (1 + \kappa_1^2)
                - \kappa_1 \sum f_i - \kappa_1 \sum \kappa_i^2 f_i.
\end{equation}
\end{lemma}

By Lemma \ref{lem2.10} we immediately derive

\begin{lemma}
\label{lem4.30}
Let $\Sigma$ be a smooth hypersurface in $\bfH^{n+1}$ satisfying
$f(\kappa)=\sigma$. Then in a local orthonormal frame,
\begin{equation}
\label{eq4.60}
  F^{ij} \nabla_{ij} \nu^{n+1} = \frac2{u} F^{ij}\nabla_i u \nabla_j \nu^{n+1}
+\sigma(1+ (\nu^{n+1})^2)- \nu^{n+1}\Big(\sum f_i+\sum f_i \kappa_i^2\Big).
\end{equation}
\end{lemma}

Using Lemma \ref{lem4.20} and Lemma \ref{lem4.30} we find from \eqref{eq4.45}
\begin{equation}
\label{eq4.70}
\begin{aligned}
  0 \geq \,& -F^{ij,rs}h_{ij1}h_{rs1}+\sigma \Big(1+\kappa_1^2
  -\frac{1+(\nu^{n+1})^2}{\nu^{n+1}-a} \kappa_1\Big)\\
  +\, & \frac{a\kappa_1}{\nu^{n+1}-a} \Big(\sum f_i +\sum \kappa_i^2 f_i\Big)
  -\frac{2\kappa_1}{\nu^{n+1}-a} F^{ij}\frac{u_i}{u} \nabla_j \nu^{n+1}.
   \end{aligned}
 \end{equation}
 Next we use an inequality due
to Andrews~\cite{Andrews94} and Gerhardt~\cite{Gerhardt96} which states
\begin{equation}
\label{eq4.80}
 - F^{ij,kl} h_{ij1} h_{kl,1}
 \geq  \sum_{i \neq j} \frac{f_i - f_j}{\kappa_j - \kappa_i} h_{ij1}^2
 \geq 2 \sum_{i\geq 2} \frac{f_i - f_1}{\kappa_1 - \kappa_i} h_{i11}^2.
\end{equation}

Recall that (see \eqref{eq2.180})
  \[\nabla_i \nu^{n+1}=\frac{u_i}u (\nu^{n+1}-\kappa_i).\]
Then at $\vx_0$,  we obtain from \eqref{eq4.40}
\begin{equation} \label{eq4.90}
  h_{11i}  = \frac{\kappa_1}{\nu^{n+1}-a}\frac{u_i}u (\nu^{n+1}-\kappa_i).
  \end{equation}
 Inserting this into \eqref{eq4.80} we derive
\begin{equation}
\label{eq4.100}
 - F^{ij,kl} h_{ij1} h_{kl,1}
  \geq 2 \Big(\frac{\kappa_1}{\nu^{n+1}-a}\Big)^2
         \sum_{i\geq 2} \frac{f_i-f_1}{\kappa_1-\kappa_i}
         \,\frac{u_i^2}{u^2}(\kappa_i-\nu^{n+1})^2.
 \end{equation}

Note that we may write
\begin{equation} \label{eq4.110}
\sum f_i +\sum \kappa_i^2 f_i=(1-(\nu^{n+1})^2)\sum f_i
+\sum (\kappa_i-\nu^{n+1})^2 f_i +2\sigma \nu^{n+1}.
\end{equation}

Combining \eqref{eq4.80}, \eqref{eq4.100} and \eqref{eq4.110} gives
\begin{equation}
\label{eq4.120}
\begin{aligned}
0 \geq \,& \sigma \Big(1 + \kappa_1^2
           - \frac{(1+(\nu^{n+1})^2)}{\nu^{n+1}-a} \kappa_1\Big) \\
       \,& + \frac{a\kappa_1}{\nu^{n+1}-a} \Big((1-(\nu^{n+1})^2) \sum f_i
           + \sum (\kappa_i-\nu^{n+1})^2 f_i +2\sigma \nu^{n+1}\Big)\\
       \,& + 2\frac{\kappa_1}{\nu^{n+1}-a}
             \sum f_i \frac{u_i^2}{u^2} (\kappa_i-\nu^{n+1})\\
       \,& + 2\frac{\kappa_1^2}{(\nu^{n+1}-a)^2}\sum_{i\geq 2}
             \frac{f_i-f_1}{\kappa_1-\kappa_i} \frac{u_i^2}{u^2}
             (\kappa_i-\nu^{n+1})^2.
     \end{aligned}
 \end{equation}

Note that (assuming $\kappa_1 \geq \frac2{a}$) all the terms of
\eqref{eq4.120} are positive except possibly the ones in the sum involving
$(\kappa_i-\nu^{n+1})$ and only if $\kappa_i <\nu^{n+1}$.

Therefore define
\[  \begin{aligned}
J & = \{i:  \kappa_i - \nu^{n+1} < 0,
            \; f_i < \theta^{-1} f_1\}, \\
L & = \{i:  \kappa_i - \nu^{n+1} < 0,
            \; f_i \geq \theta^{-1} f_1\},
  \end{aligned} \]
where $\theta \in (0, 1)$ is to be chosen later.
Since $\sum u_i^2/u^2 = |\tilde{\nabla} u|^2 = 1 - (\nu^{n+1})^2 \leq 1$
and $\kappa_1 f_1 \leq \sigma$, 
we have
\begin{equation}
\label{eq4.140}
 \sum_{i \in J} (\kappa_i - \nu^{n+1}) f_i \frac{u_i^2}{u^2}
  \geq -\frac{n}{\theta} f_1 \geq -\frac{n\sigma}{\theta \kappa_1}.
\end{equation}

Finally,
\begin{equation}
\label{eq4.150}
\begin{aligned}
 \frac{2 \kappa_1^2}{\nu^{n+1}-a}
     \,& \sum_{i\in L} \frac{f_i - f_1}{\kappa_1 - \kappa_i}
         \frac{u_i^2}{u^2}(\kappa_i-\nu^{n+1})^2 \\
\geq \,& \frac{2 (1 - \theta) \kappa_1}{\nu^{n+1}-a}
         \sum_{i \in L} (\kappa_i - \nu^{n+1})^2 f_i \frac{u_i^2}{u^2} \\
   = \,& - 2 \kappa_1 \sum_{i\in L}f_i \frac{u_i^2}{u^2}(\kappa_i-\nu^{n+1})
         - \frac{2\theta \kappa_1}{\nu^{n+1}-a}
         \sum_{i \in L} (\kappa_i - \nu^{n+1})^2 f_i \frac{u_i^2}{u^2} \\
     \,& + \frac{2 \kappa_1}{\nu^{n+1}-a} \sum_{i\in L} f_i
         \frac{u_i^2}{u^2}(\kappa_i^2-(a+\nu^{n+1})\kappa_i+a\nu^{n+1})\\
\geq \,& - 2\kappa_1 \sum_{i\in L} f_i \frac{u_i^2}{u^2}(\kappa_i-\nu^{n+1})
         - \frac{2\theta}{a} \kappa_1 \sum_{i \in L} (\kappa_i - \nu^{n+1})^2
         f_i -\frac{6\sigma}{a}\kappa_1.
\end{aligned}
\end{equation}
In deriving the last inequality in \eqref{eq4.150} we have used that
$\kappa_i  f_i \leq \sigma$ for each $i$ and that $\nu^{n+1} \geq 2a$.
We now fix $\theta=\frac{a^2}4$.
From (\ref{eq4.140}) and (\ref{eq4.150}) we see that the right hand side of
\eqref{eq4.120} is strictly positive provided that $\kappa_1 > \frac{4n}{a^2}$,
completing the proof of Theorem \ref{th4.10}.
\end{proof}

\section{Uniqueness and foliations}
\label{sec5}
\setcounter{equation}{0}

In this section we identify a class of curvature functions for which there
is uniqueness. This implies that for these curvature functions and smooth
asymptotic boundaries $\Gamma$ which are Jordan, there is a foliation of
each component of $\bfH^{n+1}\setminus \mathcal{C}(\Gamma)$ (the complement
of the hyperbolic convex hull of $\Gamma$) by solutions $f(\kappa)=\sigma$
as $\sigma$ varies between 0 and 1.\\

\begin{theorem}
\label{th5.10}
Let $f(\kappa)$ satisfy \eqref{eq1.50}-\eqref{eq1.100} in the positive cone
$K_n^{+} $ and in addition satisfy
\begin{equation}
\label{eq5.10}
\sum_i f_i > \sum \lambda_i^2 f_i \,\,\mbox{in  $K_n^{+} \cap \{0<f<1\}$}.
\end{equation}
Let $\Sigma_i, \, i=1,2 $ be strictly locally convex hypersurfaces
(oriented up) in $\bfH^{n+1} $ satisfying
$f (\kappa) = \sigma_i \in (0,1), \, \sigma_1 \leq \sigma_2$,
with the same boundary in the horosphere $x_{n+1}=\epsilon$ or with the same
asymptotic boundary $\Gamma=\partial \Omega$. Then $\Sigma_2$ lies below
$\Sigma_1$, that is,  if $\Sigma_i$ are represented as graphs $x_{n+1}=u_i(x)$
over $\Omega \subset R^n$, then  $u_2 \leq u_1 \,\,\mbox{in $\Omega$}$.
\end{theorem}

\begin{proof}
We build on an idea of Schlenker \cite{SCH}.
Suppose for contradiction that $\Sigma_2$ contains points in the unbounded
region of $\bfR^{n+1}_+ \setminus \Sigma_1$ and let $P$ be a point of
$\Sigma_2$ farthest from $\Sigma_1$ (necessarily $P$ is not a boundary point)
where the maximal distance, say $t^*$ is achieved. Then the local parallel
hypersurfaces $\Sigma_2^t $ to $\Sigma_2$ obtained by moving a distance $t$
(on the concave side of $\Sigma_2$ near $P$) {\em are convex} and contact
$\Sigma_1$ at a point $Q$ in $\Sigma_1$ when $t=t^*$. Moreover
$\Sigma_2^{t^*}$ {\em locally lies below $\Sigma_1$} by the maximality
of the distance $t^*$. We claim that the distance function $d(x,\Sigma_2)$
is smooth in a neighborhood of $Q$. To show this we need only show
(see \cite{LN}) that $P$ is the unique closest point to $Q$ on $\Sigma_2$.
If $P^{\prime}$ was a second point of $\Sigma_2$ at distance $t^*$
from $Q$, then the local parallel hypersurfaces $\Sigma_2^t $ to $\Sigma_2$
obtained by moving a distance $t$ (on the concave side of $\Sigma_2$ near
$P^{\prime}$) are also convex and when $t=t^*$, contact $\Sigma_1$ at $Q$
{\em and also locally lies below $\Sigma_1$} by the previous argument. This
is clearly impossible since $\Sigma_1$ has a unique tangent plane at $Q$.\\

The principal curvatures of $\Sigma_2^t$ at points along the normal geodesic
emanating from any point of $\Sigma_1$ (say near $P$) are given by the ode
(see \cite{Gray}):
\[ \kappa_i^{\prime}(t) = 1-\kappa_i^2.\]
In particular, if $\kappa_i(0)< 1$, then $\kappa_i(0)\leq \kappa_i(t)<1$
while if $\kappa_i(0)>1$, then $1 <\kappa_i(t) \leq \kappa_i(0)$.
Of course if $\kappa_i(0)=1$ , then $\kappa_i(t)\equiv 1$.
Moreover by \eqref{eq5.10},
\begin{equation}
\label{eq5.30}
\frac{d}{dt} f(\kappa)(t)=\sum f_i-\sum k_i^2 f_i >0
\,\,\mbox{in  $K_n^{+} \cap \{0<f<1\}$}.
\end{equation}

It follows that the $\Sigma_2^t$ satisfy $f(\kappa)>\sigma_2$ and so
are strict subsolutions of the equation $f(\kappa)=\sigma_1$.
On the other hand at $t=t*$ we have $\Sigma_2^t$ lies below $\Sigma_1$
but touches $\Sigma_1$ at $Q$ violating the maximum principle.
\end{proof}

\begin{corollary}
\label{cor5.20}
Let $f(\kappa)$ satisfy \eqref{eq1.50}-\eqref{eq1.100} in the positive cone
$K_n^{+} $ and in addition satisfy \eqref{eq5.10}.
Let $\Sigma_i, \, i=1,2 $ be strictly locally convex graphs (oriented up)
in $\bfH^{n+1} $over $\Omega \subset \bfR^n$ satisfying
$f(\kappa)=\sigma \in (0,1)$ with the same boundary in the horosphere
$x_{n+1}=\epsilon$ or with the same asymptotic boundary
$\Gamma=\partial \Omega$. Then $\Sigma_1=\Sigma_2$.
\end{corollary}

\begin{example}
\label{ex5.30}
For $l-1$ or $l-2$, let $f=(\frac{H_n}{H_l})^{\frac1{n-l}}$ in the cone
$K_n^+ \subset \bfR^n$. Then (see Lemma 2.14 of \cite{Spruck03})
\[ \begin{aligned}
  f_i &\, = \frac{f}{n-l} \Big(\frac1{\lambda_i}-(\log{H_l})_i\Big)
      \frac{f}{n-l} \Big(\frac1{\lambda_i}-\frac{H_{l-1;i}}{H_l}\Big),
\end{aligned} \]
where $H_{l-1;i}=H_{l-1}|_{\lambda_i=0}$.
Hence,
\be \label{eq5.40}
\sum f_i=\frac{f}{n-l} \Big(n\frac{H_{n-1}}{H_n}-l\frac{H_{l-1}}{H_l}\Big)~.
\ee
Similarly,
\[ \sum \lambda_i^2 f_i=\frac{f}{n-l}
     \Big(n H_1 -\frac{\sum \lambda_i^2 H_{l-1;i}}{H_l}\Big)~.\]
Using
\[\frac{\sum \lambda_i^2 H_{l-1;i}}{H_l} = n H_1 -(n-l)\frac{H_{l+1}}{H_l}~,\]
we find
\be \label{eq5.50}
\sum \lambda_i^2 f_i=f \frac{H_{l+1}}{H_l}~.
\ee
Combining \eqref{eq5.40} and \eqref{eq5.50} gives
\be
\label{eq5.60}
\sum f_i - \sum  \lambda_i^2 f_i
   = \frac{f}{n-l}\Big(n \frac{H_{n-1}}{H_n}
     - l \frac{H_{l-1}}{H_l}-(n-l)\frac{H_{l+1}}{H_l}\Big).
\ee
By the  Newton-Maclaurin inequalities,
\[ \frac{H_{n-1}}{H_n} \geq  \frac{H_{l-1}}{H_l}\]
with equality if and only all the $\lambda_i$ are equal. Hence,
\be \label{eq5.65}
\sum f_i - \sum  \lambda_i^2 f_i
    \geq f \Big(\frac{H_{n-1}}{H_n}-\frac{H_{l+1}}{H_l}\Big)~.
\ee
Therefore if $l-1$, we find
\be \label{eq5.70}
\sum f_i - \sum  \lambda_i^2 f_i \geq 1-f^2 > 0
   \,\,\mbox{in  $K_n^{+} \cap \{0<f<1\}$}
\ee
while if $l-2$ we similarly find
\be \label{eq5.80}
\sum f_i - \sum  \lambda_i^2 f_i \geq \frac{H_{n-1}}{H_{n}} f(1-f^2)
  \geq 1-f^2 >0
\,\,\mbox{in  $K_n^{+} \cap \{0<f<1\}$}.
\ee
\end{example}

We now complete the {\bf proof of Theorem \ref{th6.0}}.
\begin{proof}
a. For $\Gamma$ smooth and $f(\kappa)$ satisfying the conditions of
Theorem~\ref{th5.10}, we have by Theorem \ref{th2.0} and Theorem~\ref{th5.10}
a smooth ``monotone decreasing'' family of smooth solutions
$\Sigma^{\sigma}=\mbox{graph \,\,$u^{\sigma}(x),\, x\in \Omega$}$ of
\eqref{eq1.10}, \eqref{eq1.20}. That is,  if  $\sigma_1< \sigma_2$, then
$u^{\sigma_1}>u^{\sigma_2}$ in $\Omega$.
Note also that if $\Omega \subset B_{\delta}(0)$ then
\[ u^{\sigma}< v^{\sigma}(x) :=
     -\frac{\sigma \delta}{\sqrt{1-\sigma^2}}
     +\sqrt{\frac{\delta^2}{1-\sigma^2}-|x|^2} \,\, \mbox{in \,\,$\Omega$},\]
where $v^{\sigma}(x)$ corresponds to the equidistant sphere solution of
$f(\kappa)=\sigma$, which is a graph over $B_{\delta}(0)$.
As $\sigma \goto 1,\, v(x) \goto 0$ uniformly and so the same holds for
$u^{\sigma}(x)$.\\

We claim  that as $\sigma \goto 0,\,\Sigma^{\sigma}$ tends to the component
$S$ of $\partial \mathcal{CH}(\Gamma)$ that is a graph over $\Omega$.
To see this note that $\Sigma^{\sigma}$ lies below $S$ but also eventually
lies above any smooth strictly locally convex hypersurface $S^{\prime}$ by
Theorem~\ref{th5.10}. \\

This completes the proof of Theorem \ref{th6.0} part a.
In order to prove part b, it suffices by a standard approximation argument,
to show that the graph solutions of $f(\kappa)=\sigma$ have uniformly bounded
principal curvatures on compact subdomains of $\Omega$, independent
of the smoothness of $\Gamma$. We carry this out for the special curvature
quotients $f=\frac{H_n}{H_{n-1}}$ in Lemma~\ref{lem5.30} below,
 thus completing the proof of part b.
\end{proof}

\begin{lemma}
\label{lem5.30}
Let $\Sigma=\mbox{\{graph $u(x): x\in \Omega$}\}$ be the unique
strictly locally convex solution
of $\frac{H_n}{H_{n-1}}(\kappa)=\sigma \in (0,1)$. For any compact subdomain
$\Omega^{\prime} \subset \subset \Omega$, let
$\Sigma^{\prime}=\mbox{\{graph $u(x): x\in \Omega^{\prime}$}\}$. Then,
\[ \max_{\vx \in \Sigma^{\prime}}  \kappa_{\max}  \leq C,\]
where C depends only on $\sigma$ and the (Euclidean) distance from
$\Omega^{\prime}$ to $\partial \Omega$.
\end{lemma}

\begin{proof}
Fix a small constant $\theta \in (0,1)$ and set $\phi=(u-\theta)_{+}$.
Recall from Lemma~\ref{lem2.10},
\be
\label{eq5.90}
Lu = \frac{2}{u}F^{ij}u_i u_j + \sigma u \nu^{n+1 }- u \sum f_i
\ee
We modify the argument of section~\ref{sec4} by setting
\be
\label{eq5.100}
 M_0 = \max_{\vx \in \Sigma}  \phi \,  \kappa_{\max} (x)~.
 \ee
 Then $M_0 > 0$ is attained at an interior point $\vx_0 \in \Sigma$.
Let $\tau_1, \ldots, \tau_n$ be a local orthonormal frame
around $\vx_0$ such that $h_{ij} (\vx_0) = \kappa_i \delta_{ij}$,
where $\kappa_1, \ldots, \kappa_n$ are the principal curvatures of $\Sigma$
at $\vx_0$. We may assume
$\kappa_1 = \kappa_{\max} (\vx_0)$.
Thus, at $\vx_0,\,  \log{\phi} +\log{h_{11}}$
 has a local maximum and so,
\begin{eqnarray}
\label{eq5.110}
\frac{\phi_i}{\phi}+\frac{h_{11i}}{h_{11}} = 0, \\
\label{eq5.120}
 \frac{\phi_{ii}}{\phi} + \frac{h_{11ii}}{h_{11}} \leq 0.
\end{eqnarray}

 As in section \ref{sec4} we obtain from \eqref{eq5.120} and \eqref{eq5.90} that
 at $\vx_0$,
 \be \label{eq5.130}
 0 \geq - \frac{u \kappa_1}{\phi} \sum f_i
+\sigma (1 + \kappa_1^2 )- \kappa_1 \sum f_i - \kappa_1 \sum \kappa_i^2 f_i.
 \ee
 From the calculations of Example \ref{ex5.30},
 \be
\label{eq5.140}
  1 \leq \sum f_i \leq n,\;\; \sum \kappa_i^2 f_i=\sigma^2\,.
  \ee
Hence from \eqref{eq5.130} and \eqref{eq5.140} we obtain
$\phi \kappa_1 \leq C$. Choosing $\theta$ so small that $u\geq 2\theta$ on
$\Omega^{\prime}$ completes the proof.
\end{proof}

\end{document}